\documentclass [12pt,oneside]{amsart}
\usepackage{amsmath}

\usepackage {amsfonts}
\usepackage {amsmath}
\usepackage {amsthm}
\usepackage {amssymb}
\usepackage {framed}
\usepackage {amsxtra}
\usepackage {enumerate}
\usepackage {graphicx}
\usepackage{color}
\usepackage{graphicx}
\usepackage[left=0.87in, right=.86 in, top=1.2 in, footskip=0.5 in, bottom=0.9 in]{geometry}
\usepackage[font=small,skip=0pt]{caption}

\makeatletter

\theoremstyle{definition}
\newtheorem{df}{Definition} [section]

\theoremstyle{plain}
\newtheorem{thm}[df]{Theorem}

\newtheorem{lemma}[df]{Lemma}
\newtheorem{cor}[df]{Corollary}

\newtheorem{obs}[df]{Observation}
\newtheorem{problem}[df]{Problem}

\title{The largest angle bisection procedure}

\author{Dan Ismailescu}
\address{Mathematics Department, Hofstra University, Hempstead, NY 11549}
\email{dan.p.ismailescu@hofstra.edu}

\author{Joehyun Kim}
\address{Fort Lee High School, NJ 07024}
\email{joehyunkim5@gmail.com}

\author{Kelvin Kim}
\address{Bergen Catholic High School, NJ 07649}
\email{kelvin2kim@gmail.com}

\author{Jeewoo Lee}
\address{Townsend Harris High School, NY 11367}
\email{jlee1397@townsendharris.org}

\usepackage{fancyhdr}
\begin{document}

\thispagestyle{empty}

\begin{abstract}

The {\it largest angle bisection} procedure is the operation which partitions a given triangle, $T$, into two smaller triangles by constructing the angle bisector of the largest angle of $T$. Applying the procedure to each of these two triangles produces a partition of $T$ into four smaller triangles. Continuing in this manner, after $n$ iterations, the initial triangle is divided into $2^n$ small triangles.
We prove that as $n$ approaches infinity, the diameters of all these $2^n$ triangles tend to $0$, the smallest angle of all these triangles is bounded away from $0$, and that, with the exception of $T$ being an isosceles right triangle, the number of dissimilar triangles is unbounded.

\end{abstract}

\maketitle
\thispagestyle{empty}
\pagenumbering{arabic}


\begin{section}{\bf Background and motivation}

For a given triangle, locate the midpoint of the longest side and then connect this point to the vertex of the
triangle opposite the longest side. In other words, in any given triangle draw the shortest median.
This construction is known as the {\it longest edge bisection} procedure and was first considered in 1975 by
Rosenberg and Stenger \cite{rosenbergstenger}.

Let $\Delta_{01}$ be a given triangle. Bisect $\Delta_{01}$ into two triangles $\Delta_{11}$ and $\Delta_{12}$ according to the procedure defined above. Next, bisect each $\Delta_{1i}$, $i=1,2$, forming four new triangles $\Delta_{2i}$, $i=1,2,3,4$. Continue in this fashion. For every nonnegative integer $n$ set $T_n=\{\Delta_{ni}: \,1\le i\le 2^n\}$, so $T_n$ is the set of $2^n$ triangles  created in the $n$-th iteration. Please refer to figure \ref{threeit} for an illustration of this process for $n=3$.
\begin{figure}[!htb]
\centering
\includegraphics[scale=0.85]{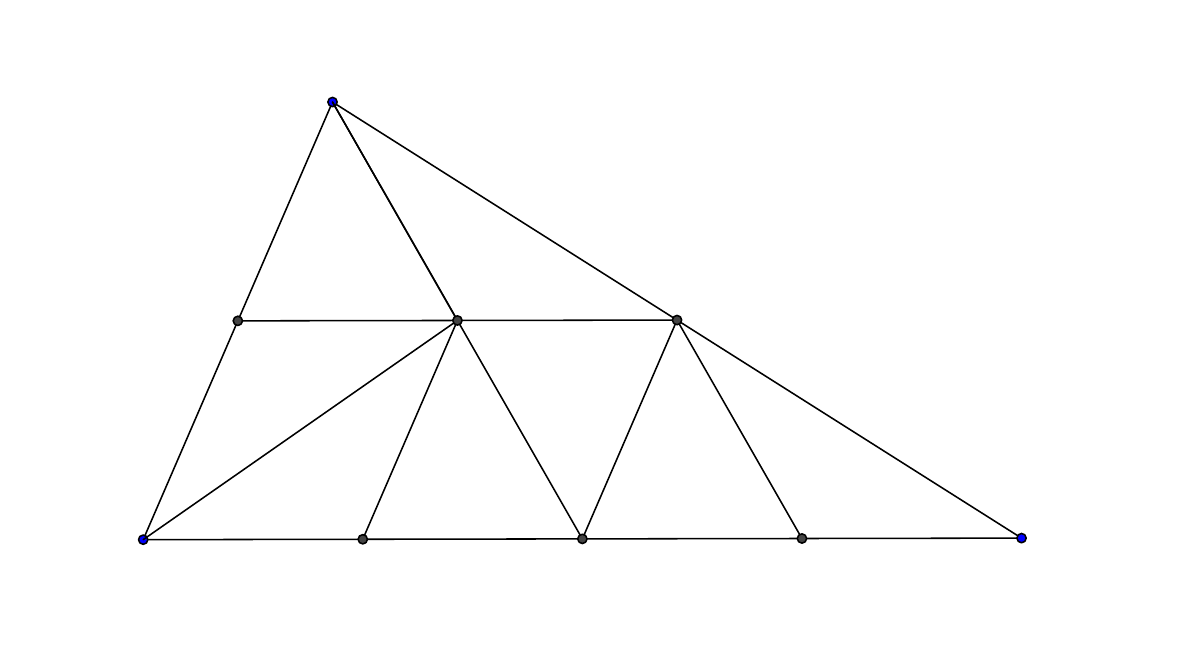}
\vspace{-0.5cm}
\caption{The longest edge bisection procedure: the first three iterations}
\label{threeit}
\end{figure}

Define $m_n$, the {\it mesh} of $T_n$, to be the length of the longest side among the sides of all triangles in $T_n$.
Similarly, let $\gamma_n$ be the smallest angle among the angles of the triangles in $T_n$.

Motivated by possible applications to the finite element method, Rosenberg and Stenger considered the following:
\begin{problem}\label{mainproblem}
\begin{itemize}
\item[(a)]{Is it true that $\gamma_n$ is bounded away from $0$ as $n\rightarrow \infty?$}
\item[(b)]{Is it true that $m_n$ approaches $0$ as $n\rightarrow \infty$?}
\item[(c)]{ Does the family $\bigcup_{n=0}^{\infty}T_n$ contains finitely many triangle types?}
\end{itemize}
\end{problem}

Based on figure \ref{threeit} it is reasonable to expect the answer to the first two questions to be affirmative.
Indeed, the first question was answered by Rosenberg and Stenger themselves.
\begin{thm}\cite{rosenbergstenger}
With the notations above we have that
\begin{equation*}
\gamma_n\ge \arctan \left(\frac{\sin{\gamma_0}}{2-\cos{\gamma_0}}\right)\ge \gamma_0/2,
\end{equation*}
where $\gamma_0$ is the smallest angle of the initial triangle $\Delta_{01}$. Equality holds when
$\Delta_{01}$ is an equilateral triangle.
\end{thm}

As mentioned earlier, the theorem is of interest if the mesh in the finite-element approximation of solutions of differential equations is refined in the described manner; the convergence criterion of the method is that the angles of the triangles do not tend to zero.

In 1890 Schwarz \cite{schwarz} surprised the mathematical community by providing and explicit example of a situation in which triangles are used to approximate the area of a cylinder. In this case, the sum of the areas of the triangles may not converge to the area of the cylinder as the size of each triangle approaches zero, and the number of triangles approaches infinity, if the smallest interior angle of each triangle approaches zero.

The second question was answered by Kearfott \cite{kearfott} a few years later.
\begin{thm}\cite{kearfott} Let $m_n$ be the length of the longest side among the sides of all $n$th generation triangles obtained
by applying the longest edge bisection procedure. Then
\begin{equation*}
m_n\le m_0\cdot \left(\frac{\sqrt{3}}{2}\right)^{\lfloor{\frac{n}{2}}\rfloor}\,\,\text{and therefore}\,\,m_n\rightarrow 0 \,\,\text{ as}\,\, n\rightarrow \infty.
\end{equation*}
\end{thm}
Kearfott shows that $m_2\le m_0\cdot(\sqrt{3}/{2})$ and then uses induction. This rate of convergence was successively improved by Stynes \cite{stynes1} and by Adler \cite{adler} who proved that $m_n\le m_0\cdot \sqrt{3}\cdot 2^{-n/2}$ if $n$ is even and $m_n\le m_0\cdot \sqrt{2}\cdot 2^{-n/2}$ if $n$ is odd, with equality if the initial triangle is equilateral.

Also, both Stynes' and Adler's techniques lead to an answer to the third question: the union $\bigcup_{n=0}^{\infty} T_n$ contains only finitely many triangle shapes (up to similarity).

For a given initial triangle $\Delta_{01}$, it would be interesting to find a formula for the number of different similarity classes generated by the longest edge bisection procedure applied to $\Delta_{01}$, and also an expression for the smallest $N$ such that every triangle in $\bigcup_{n=0}^{\infty}T_{n}$ is similar to some triangle in $\bigcup_{k=0}^{N} T_k$. At the time of this writing, there are several known bounds but these seem rather weak. For details the reader is referred to \cite{iribarren, plaza}.
\end{section}

\begin{section}{\bf The problem and summary of results}

In this paper we consider a different kind of bisection procedure.

{\bf Question.} What if instead of bisecting {\it the longest edge}, we bisect {\it the largest angle}?

For any given triangle, locate the \underline{largest angle} and then construct the \underline{angle bisector} of this angle - see figure \ref{lemma1fig} below.

\begin{figure}[!htb]
\centering
\includegraphics[scale=1]{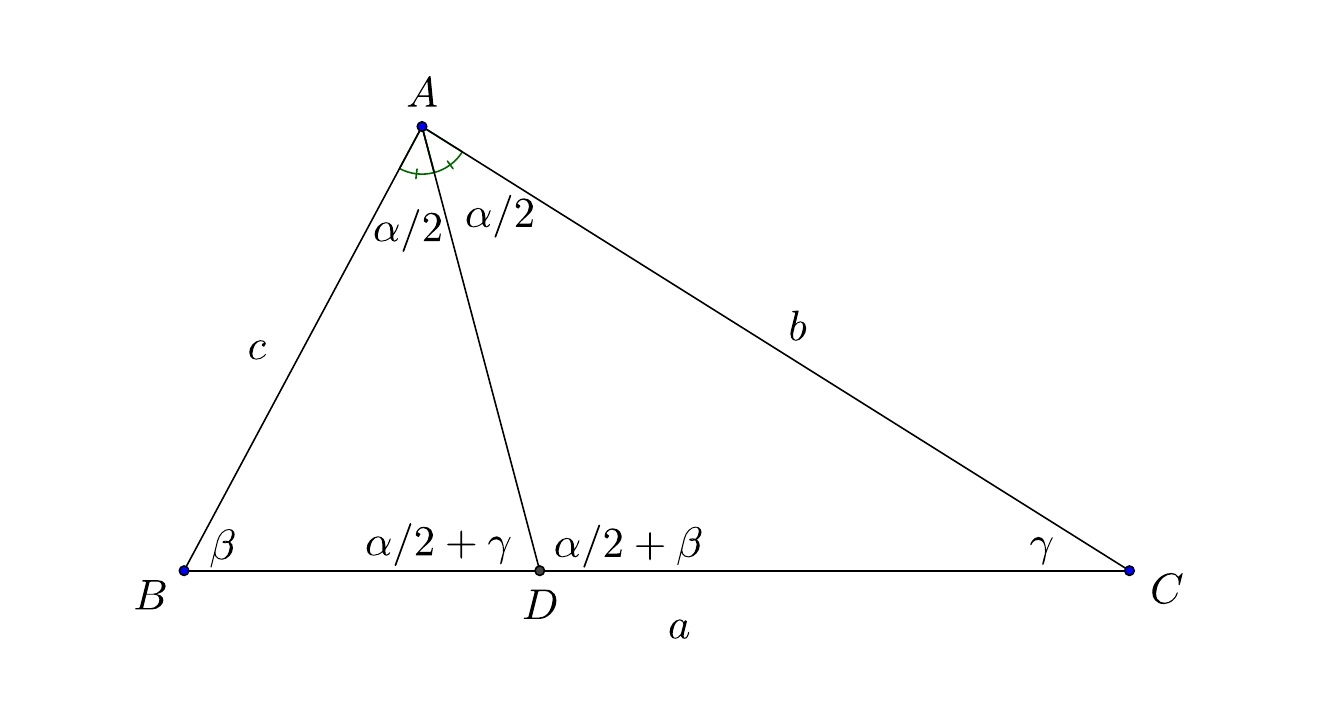}
\vspace{0cm}
\caption{The first iteration of the largest angle bisection procedure. In triangle $ABC$ we have $\alpha\ge \beta\ge \gamma$.}
\label{lemma1fig}
\end{figure}

For each of the two newly formed triangles construct the angle bisectors of their largest angles, and so on. As in the longest edge bisection scenario, let $T_n$ be the set of $2^n$ triangles obtained after the $n$th iteration of this operation, which we are going to call the {\it largest angle bisection} procedure. Also, let $m_n$, the {\it mesh} of $T_n$, to be the length of the longest side among the sides of all triangles in $T_n$ and let $\gamma_n$ be the smallest angle among the angles of the triangles in $T_n$.

It is then natural to ask the same questions as in problem \ref{mainproblem} for this new operation.  Under the assumption of the largest angle bisection procedure we prove the following results.
\begin{align}
&\gamma_n = \min(\gamma, \alpha/2 ), \,\,\text{for all} \,\, n\ge 1. \label{minangle}\\
& m_n\rightarrow 0 \,\, \text{as}\,\, n\rightarrow \infty,\label{mn}\\
&\text{With one exception, the set} \,\, \bigcup_{n=0}^{\infty} T_n \,\, \text{contains infinitely many similarity types}.\label{types}
\end{align}

Notice that results \eqref{minangle} and \eqref{mn} are similar to the ones in the original problem, while result \eqref{types} is different. The remainder of the paper is dedicated to presenting proofs of these statements.

Showing \eqref{minangle} is very easy, and the proof of \eqref{types} is not too difficult, either. However, proving \eqref{mn} is quite challenging. In fact, throughout the next three sections we build the tools needed for showing that $m_n\rightarrow 0$.
Let us start with a simple proof of \eqref{minangle}.
\begin{thm}\label{minanglethm}
Let $\Delta_{01}=ABC$ be an arbitrary triangle with angles $\alpha\ge \beta\ge \gamma$. Apply the largest angle bisection procedure
with $ABC$ as the initial triangle. Then, for all $n\ge 1$ we have that $\gamma_n= \min(\gamma, \alpha/2)$.
\end{thm}
\begin{proof}

Each of the $2^n$ triangles obtained after the $n^{th}$ iteration has a largest angle. Let $\alpha_n$ denote the smallest such angle.
It is easy to see that
\begin{equation}\label{step1}
\gamma_{n+1}\ge \min(\gamma_n, \alpha_n/2).
\end{equation}
Indeed, if $\gamma_{n+1}$ is obtained by bisecting the largest angle of some $n$-th generation triangle then $\gamma_{n+1}\ge \alpha_n/2$. Otherwise, $\gamma_{n+1}$ appears
a base angle of some $n^{th}$ generation triangle, hence, $\gamma_{n+1}\ge \gamma_n$.

Next we prove that
\begin{equation}\label{step2}
\alpha_{n+1}/2\ge \min(\gamma_n, \alpha_n/2).
\end{equation}
Let $MNP$ be the $n^{th}$ generation triangle one of whose offspring contains $\alpha_{n+1}$.
Without loss of generality we can assume that $\alpha_{n+1}$ is one of the angles of triangle $MQP$ - see figure \ref{MNP}.
\begin{figure}[!htb]
	\centering
	\includegraphics[scale=0.75]{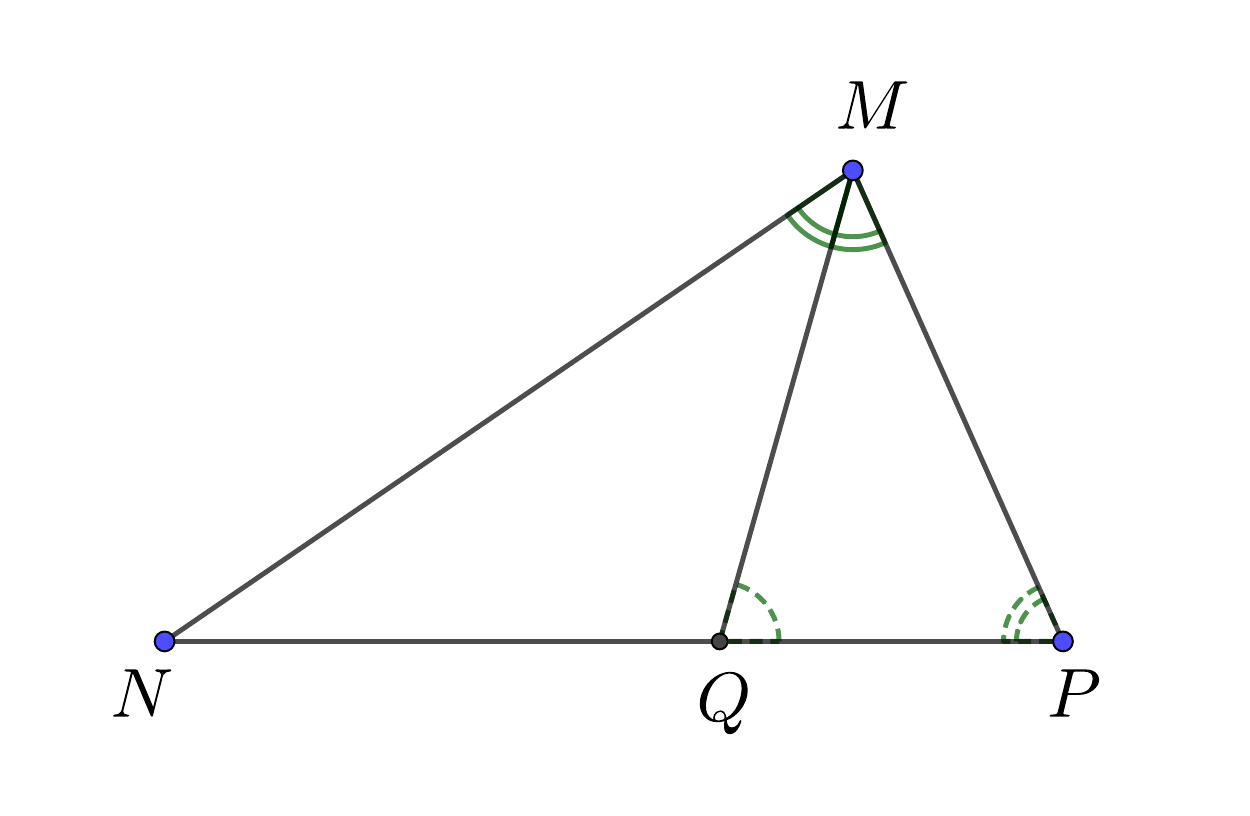}
	\vspace{-0.5cm}
	\caption{$\alpha_{n+1}=\max(\angle MQP, \angle MPQ)$}
	\label{MNP}
\end{figure}

Clearly, $\angle QMP <\angle MQP$ which implies that $\alpha_{n+1}= \max(\angle MQP,\angle MPQ) \ge \angle MQP$.

It follows that
\begin{equation*}
\alpha_{n+1}\ge \angle MQP = \angle NMQ+\angle MNQ =\angle NMP/2+\angle MNQ \ge \alpha_n/2+\gamma_n\ge \min(\alpha_n, 2\gamma_n).
\end{equation*}

Combining \eqref{step1} and \eqref{step2} we obtain that
$\min(\gamma_{n+1},\alpha_{n+1}/2)\ge \min(\gamma_n, \alpha_n/2)$, from which we obtain that
\begin{equation}\label{ineq1}
\gamma_n\ge \min(\gamma_{n-1}. \alpha_{n-1}/2)\ge \min(\gamma_{n-2}. \alpha_{n-2}/2)\ge \ldots\ge \min(\gamma_0, \alpha_0/2)=\min(\gamma, \alpha/2).
\end{equation}
On the other hand, it is easy to see that for all $n\ge 1$
\begin{equation}\label{ineq2}
\gamma_n\le \min(\gamma, \alpha/2).
\end{equation}
Indeed, if $\min(\gamma, \alpha/2)=\gamma$ then $\gamma$ appears in some $n^{th}$ generation triangle for  all $n\ge 0$ since one never bisects angles which are less than $60^{\circ}$.
In this case, it follows that $\gamma_n\le \gamma= \min(\gamma, \alpha/2)$.

Otherwise, $\min(\gamma, \alpha/2)=\alpha/2$ then $\alpha/2$ appears in some $n^{th}$ generation triangle for  all $n\ge 1$ for exactly the same reason as above.
Again, we obtain that $\gamma_n\le \alpha/2 = \min(\gamma, \alpha/2)$.
This proves inequality \eqref{ineq2}.
From \eqref{ineq1} and \eqref{ineq2} the statement of Theorem \ref{minanglethm} follows.

\end{proof}
\end{section}

\begin{section}{\bf Showing that $m_n\rightarrow 0$: Initial considerations}

Recall that in the longest edge bisection procedure it is relatively easy to prove that $m_2\le m_0\cdot \sqrt{3}/2$ and
in general that $m_{n+2}\le m_n\cdot \sqrt{3}/2$. This eventually implies that $m_n\le m_0\cdot (\sqrt{3}/2)^\lfloor{n/2}\rfloor$.
Thus $m_n\rightarrow 0$ exponentially and the base is an absolute constant - see figure \ref{altvsbis} $(a)$.
\begin{figure}[!htb]
\centering
\includegraphics[scale=0.95]{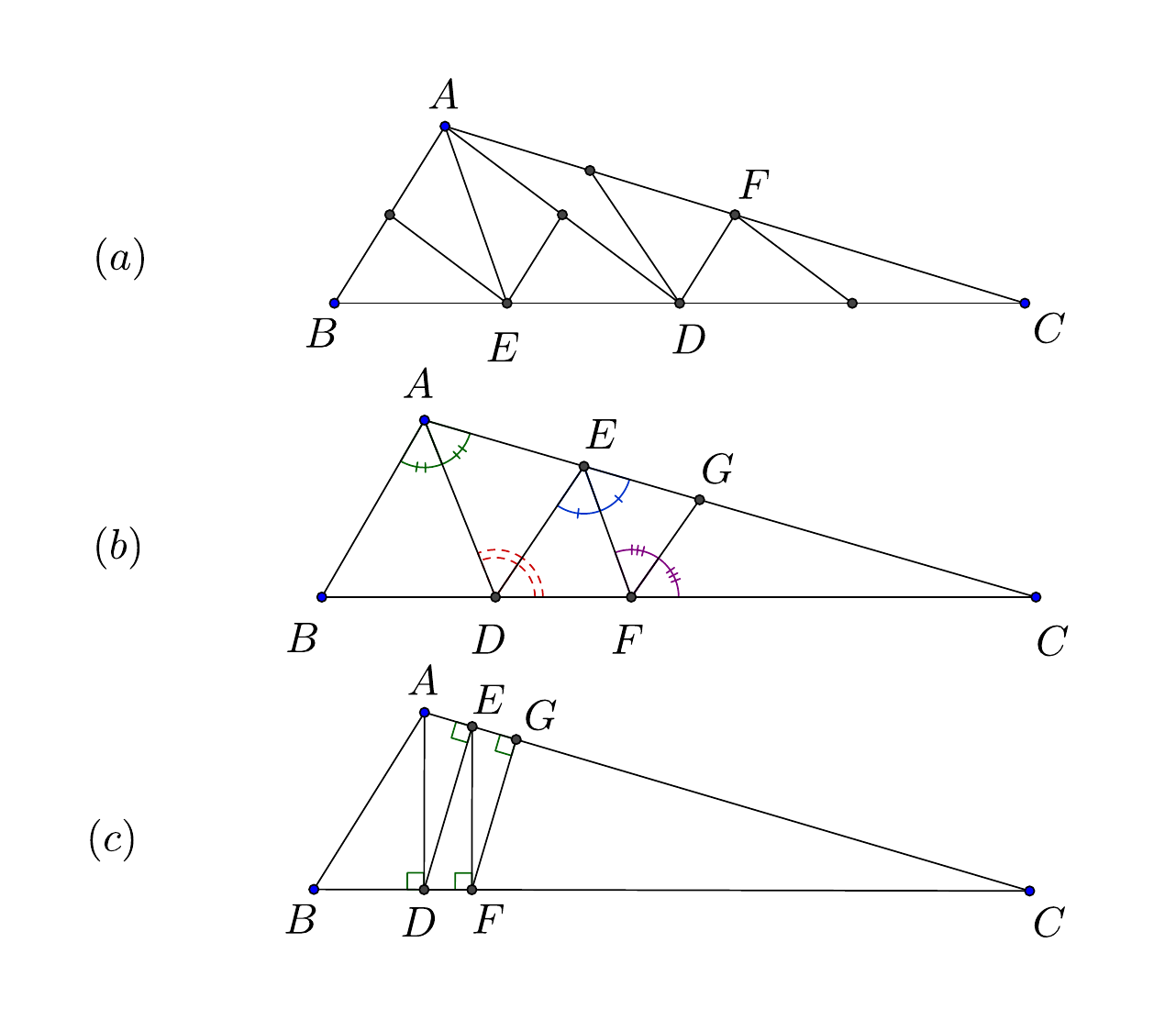}
\vspace{-0.5cm}
\caption{Mesh decay in three situations: $(a)$ the longest edge bisection; $(b)$ the largest angle bisection; $(c)$ the shortest altitude bisection}
\label{altvsbis}
\end{figure}

Note that such a result is not possible in the largest angle bisector procedure scenario.
Indeed, let $ABC$ be a very thin isosceles triangle; then the decay of $m_n$ could be quite slow
depending on the choice of $ABC$ - see figure \ref{altvsbis} $(b)$.

On the other hand, define the {\it shortest altitude bisection} procedure to be analogous to the longest
edge bisection and the largest angle bisection operation, the only difference being that at each step we
draw the altitude corresponding to the largest edge of the triangle (rather than the the median or the
angle bisector) - see figure \ref{altvsbis} $(c)$.

It it easy to see that in this case we have at most two similarity classes. Moreover, $m_n\rightarrow 0$ exponentially. Indeed,
let $ABC$ be a triangle and let $AD$ be the altitude corresponding to its longest edge. Focus of triangle $ACD$ first: denote $AD=x$, $CD=y$ and $AC=z$. Then construct $DE\bot AC$. Each of the two new triangles $ADE$ and $CDE$ is similar to $ACD$ and the corresponding similarity ratios are $x/z$ and $y/z$, respectively.

It follows that if one continues applying the shortest altitude bisection procedure to the \\subtriangles of $ACD$, the largest segment among all $n$th generation triangles cannot exceed $z\cdot \max(x/z,y/z)^n$. A similar reasoning applies to triangle $ABD$. This shows that $m_n$ approaches $0$ exponentially, but the base of this exponential depends on the initial triangle $ABC$.

It is therefore reasonable to expect that in the largest angle bisection situation, the mesh $m_n$ is going to behave in a
similar fashion. One only needs to select an appropriate quantity $k(ABC)<1$ that depends on $ABC$ and which will eventually
allow us to prove that $m_n\le m_0\cdot k^n$.

Denote $BC=a$, $AC=b$ and $AB=c$, the lengths of the sides of triangle $ABC$, Further assume that $a\ge b\ge c$.
One natural choice for $k$ would be the ratio $a/(b+c)$; we are going to call this the {\it aspect ratio} of triangle $ABC$ and
we will denote it by $r(ABC)$.  By triangle inequality, $r(ABC)<1$ so everything is fine.

The problem however is that the aspect ratio of one of the triangles obtained by largest angle bisecting $ABC$
could be greater than the aspect ratio of $ABC$. This is going to create difficulties when attempting to use induction.
So, we may have to adjust our selection of $k(ABC)$ as follows $k(ABC)=\max(r(ABC),r(ABD),r(ACD))$.

But even this is not sufficient as it may happen that there is a triangle in the second generation whose aspect ratio exceeds
the aspect ratios of all its ancestors. This is the case when $ABC$ is equilateral: $r(ABC)=1/2$,
$r(ABD)=r(ACD)=\sqrt{3}-1=0.732\ldots$ while the $(30^{\circ},45^{\circ},105^{\circ})$ triangles obtained
after the second iteration have aspect ratio $\sin{52.5^\circ}\cdot \sec{7.5^{\circ}} =0.8001\ldots$.

Fortunately, this is as far as we will have to go. At the heart of the entire proof of $m_n\rightarrow 0$ lies the following idea

{\emph {Given a triangle $ABC$, let $ABD$ and $ACD$ be the children of $ABC$ obtained via the largest angle bisection
procedure. Consider the quantity
\begin{equation*}
\rho_0= \max(r(ABC), r(ABD), r(ACD), \sqrt{3}/2).
\end{equation*}
Then all triangles obtained in the subsequent iterations have aspect ratio no greater than $\rho_0$.}}

The next two sections contain the technical details.
\end{section}

\begin{section}{\bf One simple lemma}
\begin{df}
Given a triangle $ABC$ with sides $a\ge b \ge c$, and angles $\alpha\ge \beta \ge \gamma$, define the {\it aspect ratio of $ABC$} as
\begin{equation}\label{ratiosides}
r(ABC):=\frac{a}{b+c}
\end{equation}
\end{df}

Hence, the aspect ratio of a triangle is obtained by dividing the length of the longest side by the sum of the lengths \
of the other two sides.  Obviously, an easy consequence of triangle inequality is that for any triangle we have that $r < 1$.
On the other hand, $r\ge 1/2$, with equality if and only if the triangle is equilateral.

Note that $r(ABC)$ can be expressed in terms of the angles of the triangle.

{
\begin{equation}\label{ratiotrig}
r(ABC)=\frac{a}{b+c}=\frac{\sin{\alpha}}{\sin{\beta}+\sin{\gamma}}=
\frac{2\sin{\frac{\alpha}{2}\cos{\frac{\alpha}{2}}}}{2\sin{\frac{\beta+\gamma}{2}}\cos{\frac{\beta-\gamma}{2}}}
=\sin{\frac{\alpha}{2}}\cdot\sec{\frac{\beta-\gamma}{2}}.
\end{equation}
}
\bigskip

Thus, the aspect ratio of a triangle is the product between the sine function applied to half the largest
angle and the secant function applied to half the difference of the other two angles.
Due to the nature of the problem, we are going to use \eqref{ratiotrig} much more often than \eqref{ratiosides}.

\begin{lemma}\label{lemma1}
Let $ABC$ be a triangle with sides $BC=a$, $AC=b$, $AB=c$ with $a\ge b\ge c$.
Denote the corresponding angles by $\alpha$, $\beta$ and $\gamma$, respectively.
Obviously, $\alpha \ge \beta \ge \gamma$. Let $AD$ be the angle bisector of angle $\angle BAC$.
Then the following inequalities hold true:
\begin{align}
& \frac{AD}{BC}\le \frac{\sqrt{3}}{2}.\label{ADoverBC}\\
& r(ABD)\le r(ACD).\label{rr}
\end{align}
\end{lemma}
\vspace{-0.6cm}

\begin{proof}
It is easy to express the length of the angle bisector $AD$ in terms of the side lengths $a$, $b$ and $c$. We have:
\begin{equation*}
\frac{AD^2}{BC^2}=\frac{bc}{(b+c)^2}\cdot\frac{(b+c)^2-a^2}{a^2}\le\frac{bc}{(b+c)^2}\cdot\frac{(a+a)^2-a^2}{a^2}\le \frac{1}{4}\cdot 3 =\frac{3}{4}.
\end{equation*}

This proves the first part. Since $\angle ADC$ is the largest angle of triangle $ACD$ it follows that
\begin{equation*}
r(ACD)=\sin\frac{\alpha+2\beta}{4}\cdot \sec\frac{\alpha-2\gamma}{4}.
\end{equation*}
For triangle $BCD$ we have
\[
 r(BCD) =
  \begin{cases}
   \sin\frac{\alpha+2\gamma}{4}\cdot \sec\frac{\alpha-2\beta}{4} & \text{if } \alpha/2+\gamma \geq \beta \\
   \sin\frac{\beta}{2}\cdot \sec\frac{\gamma}{2}       & \text{if } \alpha/2+\gamma \leq \beta
  \end{cases}
\]

In the first case, inequality \eqref{rr} is equivalent to
\begin{align*}
r(ACD)\ge r(ABD) &\longleftrightarrow \sin\frac{\alpha+2\beta}{4}\cdot \sec\frac{\alpha-2\gamma}{4}\ge \sin\frac{\alpha+2\gamma}{4}\cdot \sec\frac{\alpha-2\beta}{4}\longleftrightarrow\\
&\longleftrightarrow \sin\frac{\alpha+2\beta}{4}\cos\frac{\alpha-2\beta}{4}\ge  \sin\frac{\alpha+2\gamma}{4}\cos\frac{\alpha-2\gamma}{4}\longleftrightarrow\\
&\longleftrightarrow \sin\frac{\alpha}{2}+\sin{\beta}\ge \sin\frac{\alpha}{2}+\sin{\gamma}\longleftrightarrow \sin{\beta}\ge \sin{\gamma} \longleftrightarrow \beta\ge \gamma,
\end{align*}
the last step being true since $\gamma\le \beta\le 90^{\circ}$.

In the second case, inequality \eqref{rr} can be written equivalently as
\begin{align*}
&r(ACD)\ge r(ABD) \longleftrightarrow \sin\frac{\alpha+2\beta}{4}\cdot \sec\frac{\alpha-2\gamma}{4}\ge \sin\frac{\beta}{2}\cdot \sec\frac{\gamma}{2}\longleftrightarrow\\
&\longleftrightarrow \sin\frac{\alpha+2\beta}{4}\cos\frac{2\gamma}{4}\ge  \sin\frac{2\beta}{4}\cos\frac{\alpha-2\gamma}{4}\longleftrightarrow\\
&\longleftrightarrow \sin\frac{\alpha+2\beta+2\gamma}{4}+\sin\frac{\alpha+2\beta-2\gamma}{4}\!\!\ge \sin\frac{\alpha+2\beta-2\gamma}{4}+\sin\frac{-\alpha+2\beta+2\gamma}{4}\longleftrightarrow \\
&\longleftrightarrow \sin\frac{\alpha+2\beta+2\gamma}{4}-\sin\frac{-\alpha+2\beta+2\gamma}{4}\ge 0\longleftrightarrow 2\sin\frac{\alpha}{4}\cdot\cos\frac{\beta+\gamma}{2}\ge 0.
\end{align*}

which is obviously true since $\beta+\gamma<180^{\circ}$. This completes the proof.
\end{proof}
\begin{obs}\label{usefulobservation}
The results proved in Lemma \ref{lemma1} are going to be used frequently throughout the rest of the paper so it is useful to restate them as follows. Inequality \eqref{ADoverBC} says that the angle bisector of the largest angle of a triangle cannot exceed $\sqrt{3}/2$ of the length of the largest side of the triangle. Inequality \eqref{rr} states that of the two triangles created after applying the largest angle bisection procedure, the one containing the smallest angle has the larger aspect ratio.
\end{obs}
\end{section}

\begin{section}{\bf The aspect ratio lemma}
We next introduce an important quantity. For every $n\ge 0$ let
\begin{equation}\label{rdef}
r_n:=\max_{1\le i\le 2^n}(r(\Delta_{ni})),
\end{equation}
that is, $r_n$ is the maximum aspect ratio over all triangles obtained after the $n$th iteration of the
largest angle bisection procedure.

With this notation we have under the premises of Lemma \ref{lemma1} that $r_0=r(ABC)$, and by \eqref{rr} $r_1=r(ACD)$.
Since these two quantities are going to be very frequently used in the sequel we list them below for easy future reference.
\begin{equation}\label{r0r1}
r_0=r(ABC)=\sin{\frac{\alpha}{2}}\cdot\sec{\frac{\beta-\gamma}{2}}\quad \text{and}\quad
r_1=r(ACD)=\sin\frac{\alpha+2\beta}{4}\cdot \sec\frac{\alpha-2\gamma}{4}.
\end{equation}

\begin{lemma}\label{mainlemma}
Given a triangle $\Delta_{01}=ABC$ with angles $\alpha\ge \beta \ge \gamma$, let $AD$ be the angle bisector of angle $\alpha$.
Construct the angle bisectors of the largest angles in each of the triangles $\Delta_{11}=ABD$ and $\Delta_{12}=ACD$.
Let $r_2=\max(r(\Delta_{21}),\,r(\Delta_{22}),\,r(\Delta_{23}),\,r(\Delta_{24}))$ be the largest of the aspect ratios of
the four smaller triangles created after the second iteration of the largest angle bisection procedure. Then
\begin{equation}\label{r2}
r_2\le \max\left(r_0,r_1,\sqrt{3}/2\right).
\end{equation}
\end{lemma}
\begin{proof}
It is easy to see that $\angle ADC$ is the largest angle of triangle $ACD$ - see figure \ref{lemma1fig}.
Hence, in the second step one has to construct $DE$, the angle bisector of $\angle ADC$. In triangle $ABD$ however, it may be that either $\angle ABD$ or $\angle ADB$ is the largest angle.
We will therefore study two cases, depending on whether $\beta\ge \alpha/2+\gamma$ or $\beta\le \alpha/2+\gamma$.

\smallskip

\framebox[1.1\width]{{\bf Case 1. $\beta\ge \alpha/2+\gamma$}} \par

Please refer to figure \ref{case1}. We have that $\beta\ge \alpha/2+\gamma \ge \beta/2+\gamma$ from
which $\beta\ge 2\gamma$. On the other hand, $\beta\ge \alpha/2+\gamma$ which implies $2\beta\ge \alpha$.
Hence in this case
\begin{equation}\label{case1ineq}
2\beta\ge \alpha\ge \beta\ge 2\gamma.
\end{equation}
{\linespread{1.2}{
Since  $\angle BAF = \alpha/2$ and $\angle BDF = \alpha/2+\gamma$ it follows from \eqref{rr} that $r(BAF)\ge r(BDF)$.
Similarly, since $\angle DAE=\alpha/2$ and $\angle DCE=\gamma$ we have that $\angle DAE \ge \angle DCE$ and by using \eqref{rr} again, $r(CDE)\ge r(ADE)$.

This eliminates from further considerations two of the four triangles obtained in the second iteration.
In order to complete the proof of this case it would suffice to show that $r(ABF)\le r_1$ and $r(CDE)\le r_1$.
Recall that $r_1=r(ACD)$.

\begin{figure}[!htb]
\centering
\includegraphics[scale=1]{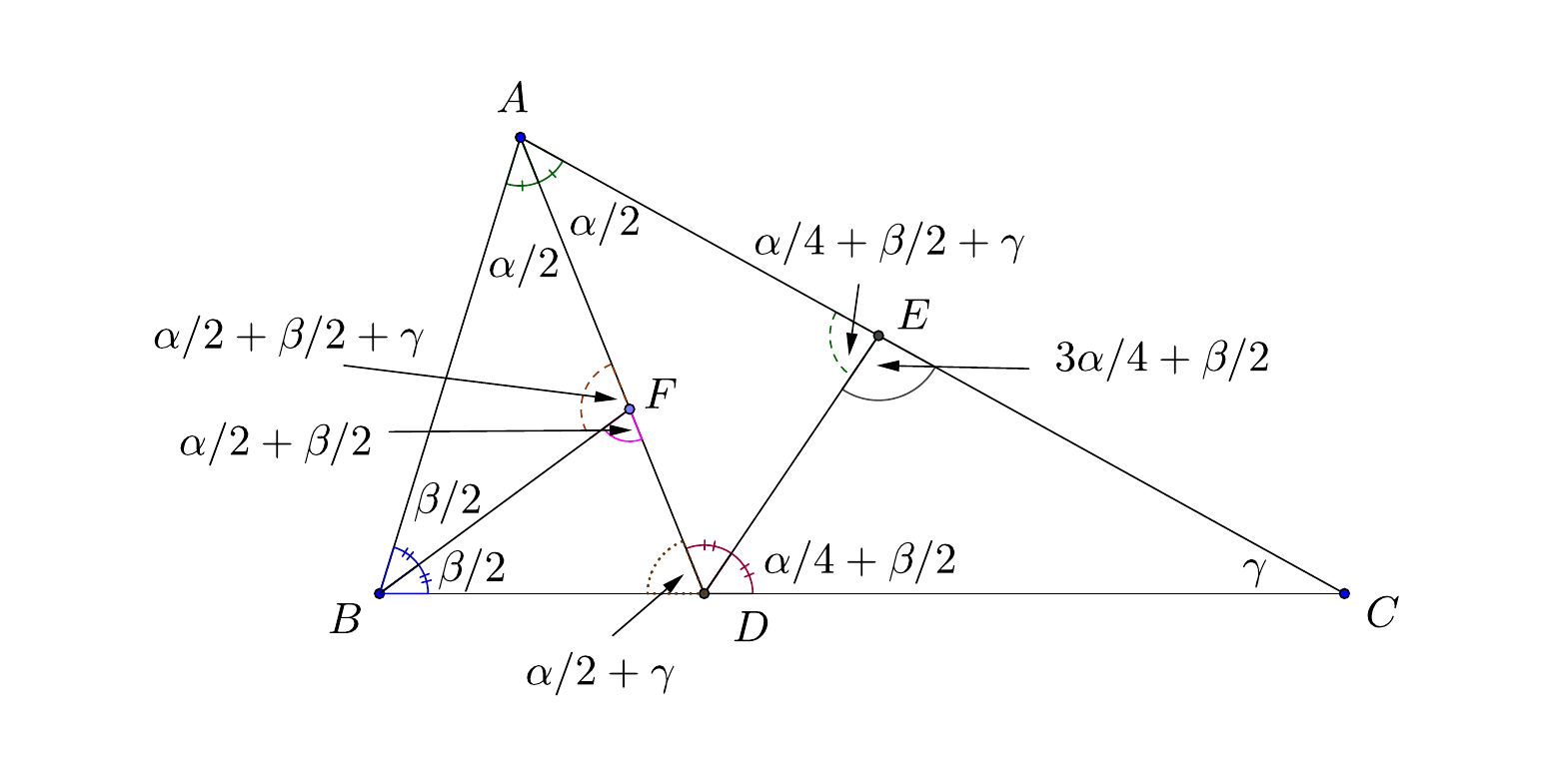}
\vspace{-0.2cm}
\caption{$\beta\ge \alpha/2+\gamma$}
\label{case1}
\end{figure}

To prove the first inequality notice the following equivalences
\begin{align*}
r(ABF)\le r(ACD)&\longleftrightarrow \sin\frac{\alpha+\beta+2\gamma}{4}\cdot \sec\frac{\alpha-\beta}{4}
\le\sin\frac{\alpha+2\beta}{4}\cdot \sec\frac{\alpha-2\gamma}{4}\longleftrightarrow\\
&\longleftrightarrow \sin\frac{\alpha+\beta+2\gamma}{4}\cdot\cos\frac{\alpha-2\gamma}{4}\le\sin\frac{\alpha+2\beta}{4}\cdot \cos\frac{\alpha-\beta}{4}\longleftrightarrow\\
&\longleftrightarrow\sin\frac{\alpha+\beta}{4}+\sin\frac{\beta+4\gamma}{4}\le \sin\frac{\alpha+\beta}{4}+\sin\frac{3\beta}{4}\longleftrightarrow\\
&\longleftrightarrow\sin\frac{3\beta}{4}-\sin\frac{\beta+4\gamma}{4}\ge 0 \longleftrightarrow \sin\frac{\beta-2\gamma}{4}\cdot \cos\frac{\beta+\gamma}{2}\ge 0,
\end{align*}
which is certainly true since $\beta\ge 2\gamma$ and $\beta+\gamma\le 180^{\circ}$.

A similar approach proves the second inequality.
\begin{align*}
&r(CDE)\le r(ACD)\longleftrightarrow \sin\frac{3\alpha+2\beta}{8}\cdot \sec\frac{\alpha+2\beta-4\gamma}{8}
\le\sin\frac{\alpha+2\beta}{4}\cdot \sec\frac{\alpha-2\gamma}{4}\longleftrightarrow\\
&\longleftrightarrow \sin\frac{3\alpha+2\beta}{8}\cdot\cos\frac{2\alpha-4\gamma}{8}\le\sin\frac{2\alpha+4\beta}{8}\cdot \cos\frac{\alpha+2\beta-4\gamma}{8}\longleftrightarrow\\
&\longleftrightarrow\sin\frac{5\alpha+2\beta-4\gamma}{8}+\sin\frac{\alpha+2\beta+4\gamma}{8}\le \sin\frac{3\alpha+6\beta-4\gamma}{8}+\sin\frac{\alpha+2\beta+4\gamma}{8}\longleftrightarrow\\
&\longleftrightarrow\sin\frac{3\alpha+6\beta-4\gamma}{8}-\sin\frac{5\alpha+2\beta-4\gamma}{8}\ge 0 \longleftrightarrow \sin\frac{2\beta-\alpha}{8}\cdot \cos\frac{\alpha+\beta-\gamma}{2}\ge 0,
\end{align*}
and that is true since $2\beta\ge \alpha$ and $\alpha+\beta-\gamma\le 180^{\circ}$.
This proves \eqref{r2} when $\beta\ge \alpha/2+\gamma$.
}}

\framebox[1.1\width]{{\bf Case 2. $\beta\le \alpha/2+\gamma$}} \par

Since $\beta\le \alpha/2+\gamma$, in order to divide triangle $ABD$,
we have to consider the angle bisector from $D$ - see figure \ref{case2}.This case is more difficult.
We need to further split the analysis into three subcases depending on whether $\alpha\ge 2\beta$,
$2\gamma\le \alpha \le 2\beta$ or $\alpha \le 2\gamma$.

\begin{figure}[!htb]
\centering
\includegraphics[scale=1]{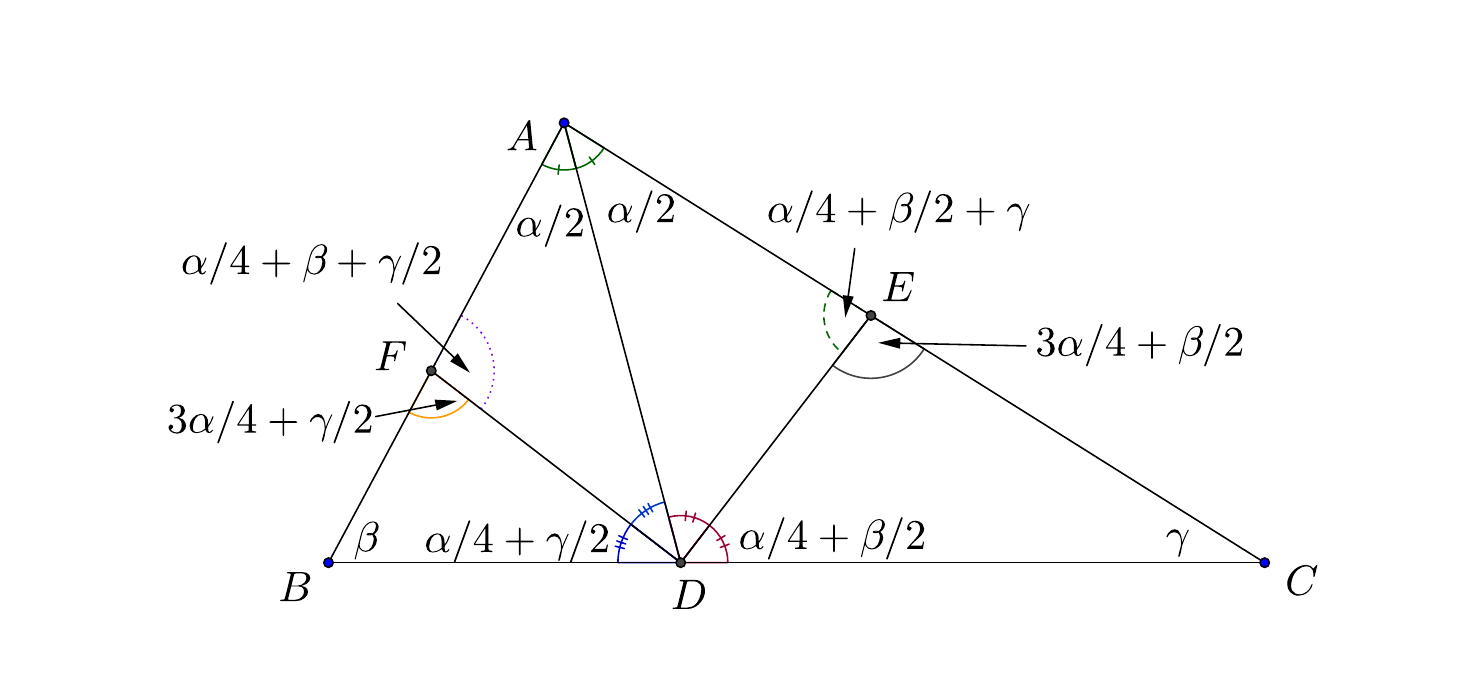}
\vspace{0cm}
\caption{$\beta\le \alpha/2+\gamma$}
\label{case2}
\end{figure}
\bigskip
\framebox[1\width]{{\bf Subcase 2.1. $\alpha\ge 2\beta$}} \par
\medskip

Using the fact that $\alpha/2\ge \beta \ge \gamma$ and \eqref{rr} we obtain that $r(BDF)\ge r(ADF)$ and
$r(CDE)\ge r(ADE)$. This removes triangles $ADE$ and $ADF$ from our analysis. To prove \eqref{r2} it would suffice
to show that $r(CDE)\le r_0$ and $r(BDF)\le r_0$. Recall that $r_0=r(ABC)=\sin\frac{\alpha}{2}\cdot\sec{\frac{\beta-\gamma}{2}}$.
 The first inequality is equivalent to
\vspace{-0.2cm}
\begin{align*}
&r(CED)\le r(ABC)\longleftrightarrow \sin\frac{3\alpha+2\beta}{8}\cdot\sec\frac{\alpha+2\beta-4\gamma}{8}\le \sin\frac{\alpha}{2}\cdot\sec\frac{\beta-\gamma}{2}\longleftrightarrow\\
&\longleftrightarrow \sin\frac{3\alpha+2\beta}{8}\cdot\cos\frac{4\beta-4\gamma}{8}\le \sin\frac{4\alpha}{8}\cdot\cos\frac{\alpha+2\beta-4\gamma}{8}\longleftrightarrow\\
&\longleftrightarrow \sin\frac{3\alpha+6\beta-4\gamma}{8}+\sin\frac{3\alpha-2\beta+4\gamma}{8}\le \sin\frac{5\alpha+2\beta-4\gamma}{8}+\sin\frac{3\alpha-2\beta+4\gamma}{8}\longleftrightarrow\\
&\longleftrightarrow \sin\frac{5\alpha+2\beta-4\gamma}{8}-\sin\frac{3\alpha+6\beta-4\gamma}{8}\ge 0\longleftrightarrow \sin\frac{\alpha-2\beta}{2}\cdot\cos{\frac{\alpha+\beta-\gamma}{2}}\ge 0,
\end{align*}
and this is true since $\alpha\ge 2\beta$ and $\alpha+\beta-\gamma<180^{\circ}$.

For the second inequality we use the following equivalences
\begin{align*}
&r(BDF)\le r(ABC) \longleftrightarrow \sin\frac{3\alpha+2\gamma}{8}\cdot\sec\frac{\alpha-4\beta+2\gamma}{8}\le \sin\frac{\alpha}{2}\cdot\sec\frac{\beta-\gamma}{2}\longleftrightarrow\\
\end{align*}
\begin{align*}
&\longleftrightarrow \sin\frac{3\alpha+2\gamma}{8}\cdot\cos\frac{4\beta-4\gamma}{8}\le \sin\frac{4\alpha}{8}\cdot\cos\frac{\alpha-4\beta+2\gamma}{8}\longleftrightarrow\\
&\longleftrightarrow \sin\frac{3\alpha+4\beta-2\gamma}{8}+\sin\frac{3\alpha-4\beta+6\gamma}{8}\le \sin\frac{5\alpha-4\beta+2\gamma}{8}+\sin\frac{3\alpha+4\beta-2\gamma}{8}\longleftrightarrow\\
&\longleftrightarrow \sin\frac{5\alpha-4\beta+2\gamma}{8}-\sin\frac{3\alpha-4\beta+6\gamma}{8}\ge 0 \longleftrightarrow \sin\frac{\alpha-2\gamma}{8}\cos\frac{\alpha-\beta+\gamma}{2}\ge 0,
\end{align*}
which is true since $\alpha\ge 2\beta\ge 2\gamma$ and $\alpha-\beta+\gamma<180^{\circ}$. This completes subcase 2.1.

\medskip

\framebox[1.1\width]{{\bf Subcase 2.2. $2\gamma\le \alpha\le 2\beta$}} \par

\medskip
Using the fact that $\alpha/2\le \beta$ and \eqref{rr} we obtain that $r(BDF)\le r(ADF)$. Similarly, using $\alpha/2\ge \gamma$ and \eqref{rr} we have that $r(CDE)\ge r(ADE)$. Thus we can safely ignore triangles $ADE$ and $BDF$ in this case. To prove \eqref{r2} it would suffice to show that $r(CDE)\le r_1$ and $r(ADF)\le r_1$. Recall that $r_1=r(ACD)=\sin\frac{\alpha+2\beta}{4}\cdot\sec{\frac{\alpha-2\gamma}{4}}$.
 The first inequality is equivalent to
\begin{align*}
&r(CDE)\le r(ACD) \longleftrightarrow \sin\frac{3\alpha+2\beta}{8}\cdot\sec\frac{\alpha+2\beta-4\gamma}{8}\le \sin\frac{\alpha+2\beta}{4}\cdot \sec\frac{\alpha-2\gamma}{4}\longleftrightarrow\\
&\longleftrightarrow \sin\frac{3\alpha+2\beta}{8}\cos\frac{2\alpha-4\gamma}{8}\le \sin\frac{2\alpha+4\beta}{8}\cos\frac{\alpha+2\beta-4\gamma}{8}\longleftrightarrow\\
&\longleftrightarrow \sin\frac{5\alpha+2\beta-4\gamma}{8}+\sin\frac{\alpha+2\beta+4\gamma}{8}\le \sin\frac{3\alpha+6\beta-4\gamma}{8}+\sin\frac{\alpha+2\beta+4\gamma}{8}\longleftrightarrow\\
&\longleftrightarrow \sin\frac{3\alpha+6\beta-4\gamma}{8}-\sin\frac{5\alpha+2\beta-4\gamma}{8}\ge 0 \longleftrightarrow \sin\frac{-\alpha+2\beta}{8}\cos\frac{\alpha+\beta-\gamma}{2}\ge 0,
\end{align*}
which is true since $\alpha\le 2\beta$ and $\alpha+\beta-\gamma<180^{\circ}$.

The second inequality is proved in a similar fashion.
\begin{align*}
&r(ADF)\le r(ACD)\longleftrightarrow \sin\frac{\alpha+4\beta+2\gamma}{8}\cdot \sec{\frac{\alpha-2\gamma}{8}}\le \sin\frac{\alpha+2\beta}{4}\cdot\sec{\frac{\alpha-2\gamma}{4}}\longleftrightarrow\\
&\longleftrightarrow \sin\frac{\alpha+4\beta+2\gamma}{8}\cos\frac{2\alpha-4\gamma}{8}\le \sin\frac{2\alpha+4\beta}{8}\cos\frac{\alpha-2\gamma}{8}\longleftrightarrow\\
&\longleftrightarrow \sin\frac{3\alpha+4\beta-2\gamma}{8}+\sin\frac{-\alpha+4\beta+6\gamma}{8}\le \sin\frac{3\alpha+4\beta-2\gamma}{8}+\sin\frac{\alpha+4\beta+2\gamma}{8}\longleftrightarrow\\
&\longleftrightarrow \sin\frac{\alpha+4\beta+2\gamma}{8}-\sin\frac{-\alpha+4\beta+6\gamma}{8}\ge 0
\longleftrightarrow \sin\frac{\alpha-2\gamma}{8}\cos\frac{\beta+\gamma}{2}\ge 0,
\end{align*}
and this is certainly valid since $\alpha\ge 2\beta\ge 2\gamma$ and $\beta+\gamma=180^{\circ}-\alpha<180^{\circ}$.
This completes the proof of subcase 2.2.

\framebox[1.1\width]{{\bf Subcase 2.3. $\alpha\le 2\gamma$}} \par
\medskip

This is the trickiest subcase.
Since $\alpha/2\le \gamma\le \beta $ it follows from \eqref{rr} that $r(BDF)\le r(ADF)$ and
$r(CDE)\le r(ADE)$. This removes triangles $BDF$ and $CDE$ from further considerations. To prove \eqref{r2} it would suffice
to show that $r(ADE)\le r(ADF)\le \sqrt{3}/2$.

The first inequality is similar to the previous ones
\begin{align*}
&r(ADF)\ge r(ADE) \longleftrightarrow \sin\frac{\alpha+4\beta+2\gamma}{8}\cdot \sec\frac{\alpha-2\gamma}{8}\ge \sin\frac{\alpha+2\beta+4\gamma}{8}\cdot\cos\frac{\alpha-2\beta}{8}\longleftrightarrow\\
&\longleftrightarrow \sin\frac{\alpha+4\beta+2\gamma}{8}\cos\frac{\alpha-2\beta}{8}\ge \sin\frac{\alpha+2\beta+4\gamma}{8}\cos\frac{\alpha-2\gamma}{8}\longleftrightarrow\\
&\longleftrightarrow \sin\frac{2\alpha+2\beta+2\gamma}{8}+\sin\frac{6\beta+2\gamma}{8}\ge \sin\frac{2\alpha+2\beta+2\gamma}{8}+\sin\frac{2\beta+6\gamma}{8}\longleftrightarrow\\
&\longleftrightarrow \sin\frac{3\beta+\gamma}{4}-\sin\frac{\beta+3\gamma}{4}\ge 0\longleftrightarrow 2\sin\frac{\beta-\gamma}{4}\cos\frac{\beta+\gamma}{2}\ge 0,
\end{align*}
and this inequality is obvious since $\beta\ge \gamma$ and $\beta+\gamma<180^{\circ}$.

It remains to show that $r(ADF)\le \sqrt{3}/2$. This proof is slightly different.
Recall that $\alpha/2\le \gamma\le \beta\le \alpha$. Denote $\beta=\gamma+x$ and
$\alpha=\gamma+x+y$ where both $x,y\ge 0$. Since $2\gamma\ge \alpha$ it follows that
$\gamma\ge x+y$ hence denote $\gamma=x+y+z$, where $z\ge 0$. To this end we have the following
\begin{equation*}
\gamma=x+y+z,\quad \beta=2x+y+z,\quad \text{and}\,\,\alpha=2x+2y+z,\qquad \text{where}\,\,x\ge 0, y\ge 0, z\ge 0.
\end{equation*}

One can express $r(ADF)$ in terms of the new variables $x$, $y$ and $z$ as follows
\begin{equation}\label{ADF}
r(ADF)= \sin\frac{\alpha+4\beta+2\gamma}{8}\cdot\sec{\frac{\alpha-2\gamma}{8}}=\sin\frac{12x+8y+7z}{8}\cdot\sec{\frac{z}{8}}
\end{equation}

Since $\alpha+\beta+\gamma=180^{\circ}$ it follows that $5x+4y+3z=180^{\circ}$ which after multiplying both sides by $2.4$ gives
$12x+9.6y+7.2z=432^{\circ}$. From here we obtain that

\begin{equation*}
\frac{12x+8y+7z}{8}\le \frac{12x+9.6y+7.2z}{8}=54^{\circ}.
\end{equation*}
On the other hand, from $5x+4y+3z=180^{\circ}$ we readily obtain $z/8\le 7.5^{\circ}$.
Using the last two inequalities in \eqref{ADF} it follows that
\begin{equation*}
r(ADF)=\sin\frac{12x+8y+7z}{8}\cdot\sec{\frac{z}{8}}\le \sin{54^{\circ}}\cdot \sec{7.5^{\circ}}=0.8159\ldots<\sqrt{3}/2,\,\,\text{as desired}
\end{equation*}
The proof of the last subcase is complete. The main lemma is proved.
\end{proof}

We are now in position to prove a useful corollary. But let us first introduce a new quantity.
\begin{df}
With the notations above let
\begin{equation}\label{rho}
\rho_n:=\max(r_n,r_{n+1},\sqrt{3}/2)
\end{equation}
\end{df}
\begin{cor}\label{corollary}
The sequence $\{\rho_n\}_{n\ge 0}$ is decreasing. That is, $\rho_0\ge \rho_1\ge \ldots \ge \rho_n\ge \rho_{n+1}\ge \ldots$.
\end{cor}
\begin{proof}
Notice that $\rho_0\ge \rho_1$ is equivalent to $\max(r_0,r_1,\sqrt{3}/2)\ge \max(r_1,r_2,\sqrt{3}/2)$ and this is exactly
what we proved in Lemma \ref{mainlemma}. Let us show that $\rho_{n+1}\le \rho_n$. Obviously, this is equivalent to proving that $r_{n+2}\le \max(r_n,r_{n+1},\sqrt{3}/2)$.

Let $T''$ be the triangle of maximum aspect ratio obtained after the $(n+2)$-nd iteration. In other words, $r(T'')=r_{n+2}$.
Triangle $T''$ has a parent triangle $T'_1$ that was obtained after the $(n+1)$-st iteration; on its turn, $T'_1$ has a parent triangle $T$ that was created after the $n$-th iteration. Let us denote by $T'_2$ be the other triangle created by applying the largest angle bisection procedure to triangle $T$.

One can think of $T'_1$ and $T'_2$ as siblings, both offsprings of $T$. Also, $T'_1$ is the parent of $T''$ while $T'_2$ is the uncle of $T''$. Note that the position of $T''$ within $T$ is irrelevant.

Now by Lemma \ref{mainlemma} it follows that $r(T'')\le \max(r(T),r(T'_1),r(T'_2),\sqrt{3}/2)$. But clearly,\\
$r(T)\le r_n$, $r(T'_1)\le r_{n+1}$ and $r(T'_2)\le r_{n+1}$, as $T$ is an $n$-th generation triangle while both $T'_1$ and $T'_2$ were obtained after the $(n+1)$-st iteration.  It follows that
\begin{equation*}
r_{n+2}=r(T'')\le \max\left(r(T),r(T'_1),r(T'_2),\sqrt{3}/2\right)\le \max(r_n,r_{n+1},\sqrt{3}/2),
\end{equation*}
which is exactly what we wanted to prove.
\end{proof}
\end{section}

\begin{section}{\bf The mesh size lemma}

In this section we prove one intermediate result involving $m_n$, the length of the longest side of
all triangles obtained after applying the largest angle bisection procedure $n$ times.

\begin{lemma}\label{mn+2}
With the notations above, for every $n\ge 0$ we have that
\begin{equation}
\frac{m_{n+2}}{m_{n}}\le \rho_n.
\end{equation}
\end{lemma}
\begin{proof}
Consider first the case $n=0$. We want to show that $m_2\le m_0\cdot \max(r_0,r_1,\sqrt{3}/2)$. Consider the triangle $ABC$ with
sides $a\ge b \ge c$ and angles $\alpha\ge \beta\ge \gamma$. Thus $m_0=a$.

Let $AD$ be the angle bisector of angle $\alpha$. As noticed earlier in \eqref{rr}, $r_1=r(ACD)$. We have two cases depending on whether $\beta \ge \alpha/2+\gamma$ or $\beta \le \alpha/2+\gamma$ - see figure \ref{twocases}.

\begin{figure}[!htb]
\centering
\includegraphics[scale=1]{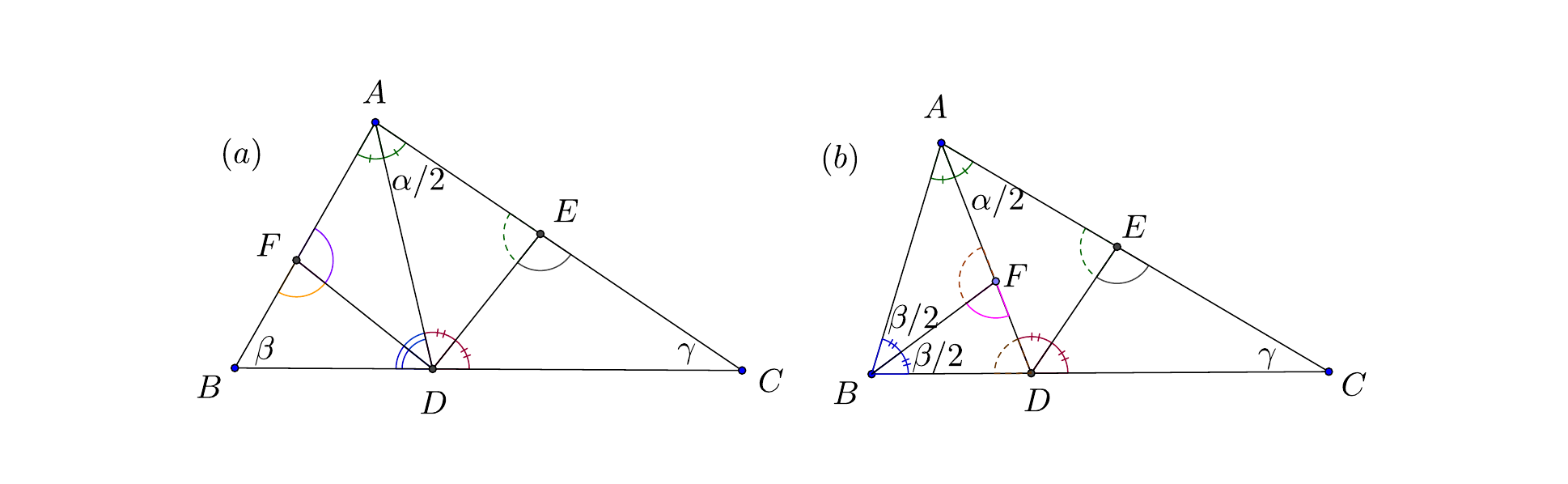}
\vspace{-0.5cm}
\caption{$(a)\,\, \beta\le \alpha/2+\gamma$,$\,\,\quad \beta\ge \alpha/2+\gamma$}
\label{twocases}
\end{figure}

In triangle $ADE$ we have that $AD$ opposes the largest angle of $\alpha/4+\beta/2+\gamma$ - see also figures \ref{case1} or \ref{case2} for a better view. Hence, $AD\ge DE$ and $AD\ge AE$ and since by \eqref{ADoverBC} we have that $AD\le BC\cdot \sqrt{3}/2$ we can safely ignore segments $DE$, $AE$ and $AD$ from future considerations. In triangle $CDE$ we have that $CD\ge CE$
since the angle opposite to $CD$ is larger. Using the angle bisector theorem in triangle $ABC$ we have that
\begin{equation*}
\frac{CD}{m_0}=\frac{CD}{BC}=\frac{AC}{AB+AC}\le \frac{BC}{AB+AC} =\frac{a}{b+c}=r_0,
\end{equation*}
and thus, all the sides of triangles $ADE$ and $CDE$ satisfy the required inequality.

Let us next look at the children of triangle $ABD$. Notice first that from the angle bisector theorem in $ABC$ we have that $BD/CD=AB/AC\le 1$, hence $BD\le CD$.

In the case when $\beta\le \alpha/2+\gamma$ it is not hard to show that
$AF\le AE$ and $BF\le CE$. For the first inequality we use the law of sines in triangles $ADF$ and $ADE$. We have
\begin{align*}
&{AF}\cdot {\sin\angle AFD}=AD \cdot \sin\angle ADF\longrightarrow AF \cdot \sin\frac{\alpha+2\gamma}{4}= AD \cdot \sin\frac{\alpha+4\beta+2\gamma}{4},\\
& AE \cdot {\sin\angle AED}=AD \cdot \sin\angle ADE\longrightarrow AE \cdot \sin\frac{\alpha+2\beta}{4}= AD \cdot \sin\frac{\alpha+2\beta+4\gamma}{4}.
\end{align*}
Combining the above equalities, the desired inequality is equivalent to
\begin{align*}
AF\le AE &\longleftrightarrow \sin\frac{\alpha+2\gamma}{4}\cdot \sin\frac{\alpha+2\beta+4\gamma}{4}\le \sin\frac{\alpha+2\beta}{4}\cdot \sin\frac{\alpha+4\beta+2\gamma}{4}\longleftrightarrow\\
&\longleftrightarrow \cos\frac{\beta+\gamma}{2}-\cos\frac{\alpha+\beta+3\gamma}{2}\le \cos\frac{\beta+\gamma}{2}-\cos\frac{\alpha+3\beta+\gamma}{2}\longleftrightarrow\\
&\longleftrightarrow \sin\beta\ge \sin\gamma,\quad \text{which is true since}\,\, \gamma\le \beta\le 90^{\circ}.
\end{align*}

The inequality $BF\le CE$ is much easier to prove. Using the angle bisector theorem again
in both triangles $ABD$ and $ACD$ we obtain
\begin{equation}
\frac{BF}{CE}=\frac{AF}{AE}\cdot\frac{BD}{CD}\le 1\cdot 1 =1.
\end{equation}
Since we proved earlier that $CE\le CD \le m_0\cdot r_0$ the proof of the first case is complete.

\bigskip

It remains to see what happens if $\beta \ge \alpha+\gamma/2$. Recall that we already dealt with the subtriangles of $ACD$.
Of the five segments appearing among the sides of triangles $ABF$ and $BDF$, $AF$ and $BF$ are clearly shorter than $AD$ and
we already know that $AD\le m_0\cdot\sqrt{3}/2$. Segment $BF$ is the angle bisector corresponding to the largest side of triangle $ABD$ hence by using \eqref{ADoverBC} again we have that $BF\le AD\cdot \sqrt{3}/{2}\le m_0\cdot 3/4$, done.
We showed earlier that $BD\le CD\le m_0\cdot r_0$. Finally, $AB\le AD$ since $\angle ABD$ is the largest one in triangle $ABD$.

It follows that $m_2\le m_0\cdot \max(r_0,r_1,\sqrt{3}/2)$ as desired.

\bigskip

The proof of the general inequality $m_{n+2}\le m_n\cdot \rho_n$ follows the same steps as the proof of Corollary \ref{corollary}.
Let $T''$ be the triangle of maximum edge length obtained after the $(n+2)$-nd iteration. In other words, $m(T'')=m_{n+2}$.
Triangle $T''$ has a parent triangle $T'_1$ that was obtained after the $(n+1)$-st iteration; on its turn, $T'_1$ has a parent triangle $T$ that was created after the $n$-th iteration. Let us denote by $T'_2$ be the other triangle created by applying the largest angle bisection procedure to triangle $T$.

One can think of $T'_1$ and $T'_2$ as siblings, both offsprings of $T$. Also, $T'_1$ is the parent of $T''$ while $T'_2$ is the uncle of $T''$. The first part of the proof implies that

\begin{equation*}
m(T'')\le m(T)\cdot \max(r(T),r(T'_1),r(T'_2),\sqrt{3}/2).
\end{equation*}

But clearly, $r(T)\le r_n$, $r(T'_1)\le r_{n+1}$ and $r(T'_2)\le r_{n+1}$, as $T$ is an $n$-th generation triangle while both $T'_1$ and $T'_2$ were obtained after the $(n+1)$-st iteration. Also, we obviously have $m(T)\le m_n$. It follows that
\begin{equation*}
m_{n+2}=m(T'')\le m(T)\cdot \max\left(r(T),r(T'_1),r(T'_2),\sqrt{3}/2\right)\le m_n\cdot\max(r_n,r_{n+1},\sqrt{3}/2)=m_n\cdot \rho_n,
\end{equation*}
which is exactly what we wanted to prove.
\end{proof}
\end{section}

\begin{section}{\bf Proofs of the last two theorems}

We are finally in position to prove that $m_n\longrightarrow 0$ as $n\longrightarrow \infty$.

\begin{thm}
Let $ABC$ be any triangle. Use the largest angle bisection procedure $n$ times with $ABC$ as the starting triangle.
Let $m_n$ and $r_n$ be the longest side and respectively, the largest aspect ratio over all $n$-th generation triangles.
Then
\begin{equation*}
 m_n\le m_0\cdot \max(r_0,r_1,\sqrt{3}/2)^{\lfloor \frac{n}{2}\rfloor}.
\end{equation*}
\end{thm}
\begin{proof}
Recall that we introduced the notation $\max(r_0,r_1,\sqrt{3}/2)=\rho_0$.
Using lemma \ref{mn+2} and corollary \ref{corollary} repeatedly we have
\begin{equation*}
m_{2n}\le m_{2n-2}\cdot \rho_{2n-2}\le m_{2n-2}\cdot \rho_0,\,\,m_{2n-2}\le m_{2n-4}\cdot \rho_{2n-4}\le m_{2n-4}\cdot \rho_0,\,\,
\ldots, m_2\le m_0\cdot \rho_0.
\end{equation*}
Multiplying term by term and simplifying we obtain that $m_{2n}\le m_0\cdot \rho_0^n$. Since $m_{2n+1}\le m_{2n}$ we also have
$m_{2n+1}\le m_0\cdot \rho_0^n$. This proves the theorem.
\end{proof}

Finally, we prove that, with one exception, the number of similarity types obtained via repeated
application of the largest angle bisection procedure is unbounded. As before, let $T_n$ be the set
of $2^n$ triangles in the $n$-th generation. Denote by
\begin{equation}\label{An}
\mathcal{A}_n=\{x\,\,|\,\, x\,\, \text{is an angle of some triangle in}\,\, T_n\}.
\end{equation}
We intend to prove that unless the initial triangle is an isosceles right triangle,
the set $\bigcup_{n=0}^{\infty} \mathcal{A}_n$ is infinite.
Let $ABC$ be a triangle with angles $\alpha\ge \beta \ge \gamma$. Apply the largest angle bisection procedure
with $ABC$ as the starting triangle. As noticed in the proof of Theorem \ref{minanglethm}, angle $\gamma$
is never bisected so it ``survives" through the entire process unscathed.

Let $\Upsilon_n$ be the $n$-th generation triangle that contains the angle $\gamma$. It turns out
one can find an explicit expression for the angles of $\Upsilon_n$ for every $n\ge 0$.

Let us introduce the {\it Jacobsthal sequence} $(j_n)_{n\ge 0}: 0,1,1,3,5,11,21,43,\ldots$ defined by the recurrence relation
$j_{n+1}=j_{n}+2j_{n-1}, j_0=0, j_1=1$. The following equalities are easy to derive
\begin{equation}\label{jacobsthal}
j_n=\frac{2^n-(-1)^n}{3},\quad j_n+j_{n+1}=2^n.
\end{equation}

Let us prove the following
\begin{lemma}\label{lemmathetaphi}
Let $\Upsilon_n$ be the triangle in $T_n$ that contains $\gamma$ as one of its angles.
Then for every $n\ge 1$ the other two angles of $\Upsilon_n$ are
\begin{equation}\label{thetaphi}
\theta_n=\frac{j_{n+1}}{2^n}\alpha+\frac{j_n}{2^{n-1}}\beta \quad \text{and}\quad \phi_n=\frac{j_{n}}{2^n}\alpha+\frac{j_{n-1}}{2^{n-1}}\beta.
\end{equation}
Moreover, $\theta_n\ge \phi_n$ and $\theta_n\ge \gamma$.
\end{lemma}
\begin{proof}
Notice first that since $(j_n)_{n\ge 0}$ is a nondecreasing sequence the inequality $\theta_n\ge \phi_n$ is immediate.
Also,
\begin{equation*}
\theta_n=\frac{j_{n+1}}{2^n}\alpha+\frac{j_n}{2^{n-1}}\beta\ge \frac{j_{n+1}}{2^n}\gamma+\frac{j_n}{2^{n-1}}\gamma= \frac{j_{n+1}+2j_n}{2^n}\gamma= \frac{2^n+j_n}{2^n}\gamma\ge \gamma.
\end{equation*}
We use induction on $n$. If $n=1$ then $\Upsilon_1 = ACD$, $\theta_1=\alpha/2+\beta$ and $\phi_1=\alpha_2$ - see figure \ref{lemma1fig}.

Suppose triangle $\Upsilon_n$ has angles $\theta_n$, $\phi_n$ and $\gamma$. Since $\theta_n$ is the largest angle of $\Upsilon_n$
we bisect this angle to obtain $\Upsilon_{n+1}$ (and one other triangle but we can ignore that one).
Then it is easy to see that the angles of $\Upsilon_{n+1}$ are $\theta_n/2+\phi_n$, $\theta_n/2$ and $\gamma$ as shown in the figure below.

\begin{figure}[!htb]
\centering
\includegraphics[scale=0.95]{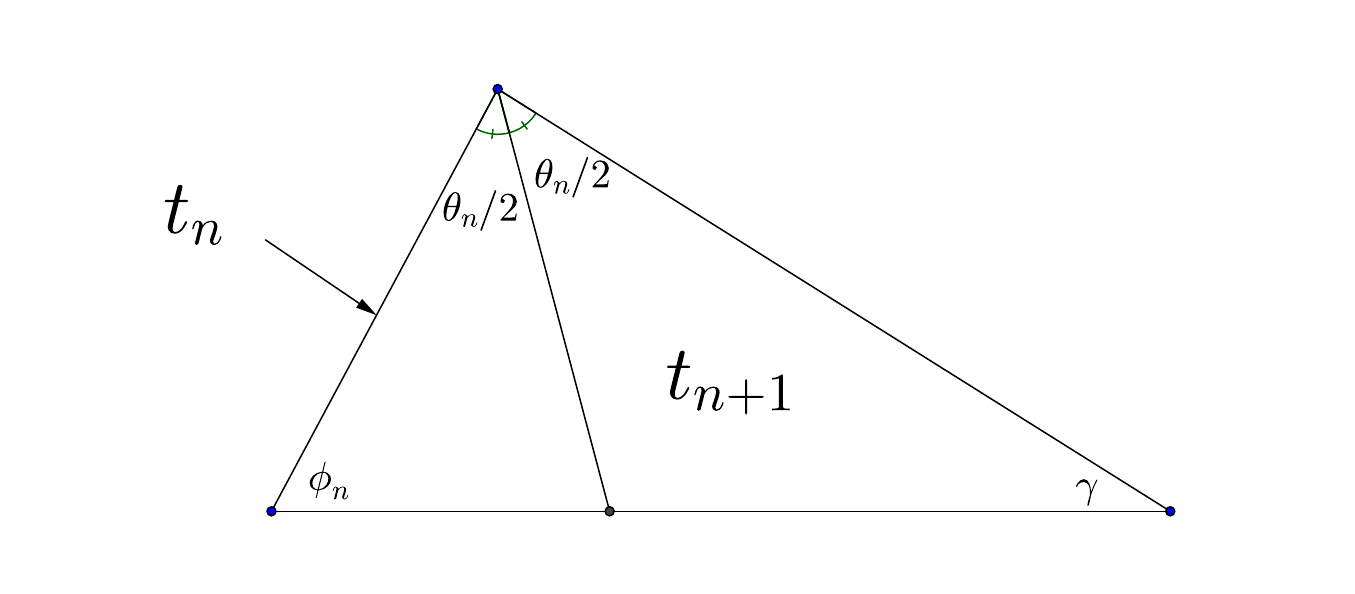}
\vspace{0cm}
\caption{Bisecting $\Upsilon_n$}
\label{Upsilon}
\end{figure}

It is now just a matter of simple algebra to verify that
\begin{equation*}
\theta_n/2+\phi_n=\frac{j_{n+1}+2j_n}{2^{n+1}}\alpha+\frac{j_n+2j_{n-1}}{2^n}\beta=
\frac{j_{n+2}}{2^{n+1}}\alpha+\frac{j_{n+1}}{2^n}\beta=\theta_{n+1},\,\,\text{and}\,\, \frac{\theta_n}{2}=\phi_{n+1}.
\end{equation*}
This completes the proof.
\end{proof}

We need one more result.
\begin{lemma}\label{last}
Let $\alpha$ and $\beta$ be two real positive numbers such that $\alpha\neq 2\beta$. For $n\ge 1$ define
\begin{equation}
\theta_n=\frac{j_{n+1}}{2^n}\alpha+\frac{j_n}{2^{n-1}}\beta. \,\, \text{Then, for all}\,\, p\neq q \,\,\text{we have}\,\, \theta_p\neq\theta_q.
\end{equation}
\end{lemma}
\begin{proof}

Notice the extra condition $\alpha\neq 2\beta$. Suppose that $p\neq q$ but $\theta_p=\theta_q$. Then we obtain
\begin{equation}\label{pq1}
\theta_p=\theta_q\longleftrightarrow \frac{j_{p+1}}{2^p}\alpha+\frac{j_p}{2^{p-1}}\beta=\frac{j_{q+1}}{2^q}\alpha+\frac{j_q}{2^{q-1}}\beta\longleftrightarrow \left(\frac{j_{p+1}}{2^p}-\frac{j_{q+1}}{2^q}\right)\alpha=\left(\frac{j_{q}}{2^{q-1}}-\frac{j_{p}}{2^{p-1}}\right)\beta.
\end{equation}
On the other hand, we have that
\begin{equation}\label{pq2}
2=\frac{j_p+j_{p+1}}{2^{p-1}}=\frac{j_q+j_{q+1}}{2^{q-1}}\longrightarrow \frac{j_{p+1}}{2^p}-\frac{j_{q+1}}{2^q}=\frac{1}{2}\cdot\left(\frac{j_{q}}{2^{q-1}}-\frac{j_{p}}{2^{p-1}}\right).
\end{equation}

Since $p\neq q$ we can divide equations \eqref{pq1} and \eqref{pq2} term by term to obtain that $\alpha=2\beta$.
But this contradicts the hypothesis. The proof is complete.
\end{proof}

We can now show that the set $\bigcup_{n=0}^{\infty}\mathcal{A}_n$ defined in \eqref{An} is, with one exception, infinite.
If $\alpha\neq 2\beta$ then by combining Lemma \ref{lemmathetaphi} and Lemma \ref{last} we have that the largest angles of
triangles $\Upsilon_n$ are all different and thus we are done.

If $\alpha=2\beta$ then the initial triangle $ABC$ is bisected into two triangles $ABD$ and $ACD$. Triangle $ACD$ has angles $2\beta, \beta, \gamma$ and it is similar to $ABC$; triangle $ABD$ has angles $\beta,\beta, 180^{\circ}-2\beta$. If the largest angle of this triangle is different from twice the middle angle then we apply the reasoning above to $ABD$ and we are done.

The only cases left to consider are those when either $\beta=2(180^{\circ}-2\beta)$ or $180^{\circ}-2\beta=2\beta$.
The first case implies $\beta=72^{\circ}$ and therefore $\alpha=144^{\circ}$ which is clearly impossible since $\alpha+\beta<180^{\circ}$. The second case gives $\beta=45^{\circ}$ and consequently, $\alpha=90^{\circ},\,\gamma=45^{\circ}$.
In this case it is obvious that the largest angle bisection procedure keeps producing isosceles right triangles so we have only one type of triangle up to similarity.

We thus proved the following

\begin{thm} Let $ABC$ be an arbitrary triangle. If $ABC$ is an isosceles right triangle then all
triangles obtained via the largest angle bisection procedure are similar to $ABC$.
Otherwise, the number of different similarity types is at least as large as the number of iterations.
\end{thm}
\end{section}
\thispagestyle{empty}

\end{document}